\documentclass[12,reqno]{amsart}
\usepackage{amsmath, amssymb, amsthm}
\usepackage{url}
\usepackage[breaklinks]{hyperref}

\DeclareMathOperator{\Z}{\mathbb{Z}}
\DeclareMathOperator{\R}{\mathbb{R}}

\theoremstyle{plain}
\newtheorem{thm}{Theorem}
\theoremstyle{definition}
\newtheorem*{defn}{Definition}

\newtheorem{lemma}[thm]{Lemma}

\begin{document}
\title[Almost consecutive partitions]{Partitions in which every term but the smallest one is consecutive}

\author{Rajat Gupta and Noah Lebowitz-Lockard}
\address{Department of Mathematics, University of Texas at Tyler, Tyler, TX 75799.}
\email{rgupta@uttyler.edu, nlebowitzlockard@uttyler.edu}
\maketitle
\begin{abstract} In this article, we introduce the notion of \emph{almost consecutive} partitions. A partition is almost consecutive if every term is consecutive, with the possible exception of the smallest one. We find formulas relating to the smallest parts of consecutive and almost consecutive partitions. We also find an alternate combinatorial interpretation of the number of almost consecutive partitions of a given integer $n$ and an asymptotic formula for this quantity.
\end{abstract}

\section{Introduction}

Let $p(n)$ be the number of partitions of $n$, i.e., representations of $n$ as an (unordered) sum of positive integers. In addition, let $p_d (n)$ be the number of partitions of $n$ into \emph{distinct} parts. The study of partitions into distinct parts has a rich history beginning with Euler's Pentagonal Number Theorem, which we write below. (Note that the $k$th pentagonal number is $k(3k - 1)/2$.)

\begin{thm}[{\cite[Cor. $1.7$]{And1}}] Let $p_o (n)$ (resp., $p_e (n)$) be the number of partitions of $n$ into an odd (resp., even) number of distinct parts. Then,
\[p_e (n) - p_o (n) = \left\{\begin{array}{ll}
(-1)^k, & \textrm{if } n = k(3k - 1)/2 \textrm{ for some integer } k, \\
0, & \textrm{otherwise.}
\end{array}\right.\]
\end{thm}

This quantity relates to the generating function of the partition function as follows:
\[\left(\sum_{n = 1}^\infty p(n) q^n\right)^{-1} = \sum_{n = 1}^\infty (p_e (n) - p_o (n)) q^n.\]
So, the inverse of the generating function of $p(n)$ is easy to understand, even though $p(n)$ itself is not. Euler also noted that this relation gives us the following recurrence for partitions.

\begin{thm}[{\cite[Cor. $1.8$]{And1}}] For all $n \geq 1$, we have
\[p(n) + \sum_{i = 1}^\infty (-1)^i \left(p\left(n - \frac{i(3i - 1)}{2}\right) + p\left(n - \frac{i(3i + 1)}{2}\right)\right) = 0.\]
\end{thm}

Uchimura later found an analogue for the Pentagonal Number Theorem. Instead of adding or subtracting $1$ depending on the parity of the number of parts, we add or subtract the smallest element. (The following theorem is more commonly attributed to Fokkink, Fokkink, and Wang \cite{FFW}, who later found an alternate proof. For yet another proof, see Andrews \cite{And2}. For further discussion and generalizations, see \cite{ABEM, GLV}.) 

From here on, we let $d(n)$ be the number of divisors of $n$. We also let $sp_o (n)$ (resp., $sp_e (n)$) be the sum of the smallest parts of the partitions of $n$ into an odd (resp., even) number of distinct parts.

\begin{thm}[{\cite{Uch}}] We have
\[sp_o (n) - sp_e (n) = d(n).\]
\end{thm}

Note that both sides of this equation have compact generating functions:
\begin{eqnarray*}
\sum_{n = 1}^\infty (sp_o (n) - sp_e (n)) q^n & = & -\sum_{n = 1}^\infty nq^n (1 - q^{n + 1})(1 - q^{n + 2}) \cdots, \\
\sum_{n = 1}^\infty d(n) q^n & = & \sum_{n = 1}^\infty \frac{q^n}{1 - q^n}.
\end{eqnarray*}
The fact that these generating functions are equal is also a special case of an identity of Ramanujan \cite[p. $354$]{Ram}. van Veen \cite{Klu} also found this equivalence in 1919, though neither he nor Ramanujan related it to the divisor function.

The smallest parts of partitions have received more attention over the past few decades, starting with Andrews' smallest part function $spt(n)$ \cite{And2}, which counts the number of appearances of the smallest parts in all partitions of $n$. 

Interestingly, Sylvester \cite[\S 46]{Syl} found another relation between partitions and divisors. (See \cite{Mas, HH} for a slight extension of this result.)

\begin{thm} The number of partitions of $n$ into a sum of distinct consecutive numbers is equal to the number of odd divisors of $n$.
\end{thm}

In this note, we relate Theorems $3$ and $4$ and prove several additional formulas. Let $sc_o (n)$ and $sc_e (n)$ be the sums of the smallest parts of the partitions of $n$ into an odd and even number of distinct consecutive parts. We show the following.

\begin{thm} Let $n$ be an integer and let $m$ be the largest odd divisor of $n$. Then
\[sc_o (n) - sc_e (n) = \frac{1}{2} (\sigma(n) + \#\{d | m : d < \sqrt{2n}\} - \# \{d | m : d > \sqrt{2n}\}).\]
\end{thm}

In addition to considering partitions into consecutive parts, we also investigate partitions which are very close to being consecutive. From here on, we  let $\#(\lambda)$ and $s(\lambda)$ be the length and smallest part of a given partition $\lambda$. We say that a partition $\lambda$ with distinct parts is \emph{almost consecutive} if every part is consecutive with the possible exception of the smallest one. For example, $2 + 6 + 7 + 8$ is an almost consecutive partition of $23$.

\begin{defn} Let $P_d (n)$ be the set of all partitions of $n$ into distinct parts, $P_c (n)$ the partitions into consecutive parts, and $P_a (n)$ the almost consecutive partitions.
\end{defn}

For example,
\begin{align*}
P_d (7) & = \{7, 6 + 1, 5 + 2, 4 + 3, 4 + 2 + 1\}, \\
P_c (7) & = \{7, 4 + 3\}, \\
P_a (7) & = \{7, 6 + 1, 5 + 2, 4 + 3\}.
\end{align*}

We show that it is possible to sum certain partition functions over the almost consecutive partitions and find a few notable results.

\begin{thm} \label{main theorem} Let $f$ be a function acting on partitions into distinct parts which only depends on the length and smallest part of a partition. (In other words, there exists a function $F: \Z_+^2 \to \R$ such that $f(\lambda) = F(\#(\lambda), s(\lambda))$ for all partitions $\lambda$ with distinct parts.) For a given integer $n \geq 3$, we have
\begin{eqnarray*}
\sum_{\lambda \in P_a (n)} f(\lambda) & = & 2\sum_{\lambda \in P_d (n)} f(\lambda) + \sum_{\lambda \in P_c (n + 1)} f(\lambda) - \sum_{\lambda \in P_d (n + 1)} f(\lambda) - \sum_{\lambda \in P_d (n - 2)} f(\lambda) \\
& & + \sum_{i = 1}^{\lfloor (n - 3)/2 \rfloor} f((i, n - i)) + f((n - 2)) - f((n)).
\end{eqnarray*}
\end{thm}

We prove Theorem \ref{main theorem} in Section $4$. In Section $5$, we examine the specific cases where $f(\lambda) = 1$, $f(\lambda) = (-1)^{\# (\lambda)}$, and $f(\lambda) = (-1)^{\# (\lambda)} s(\lambda)$. The last one gives us an analogue of Theorem $5$ for almost consecutive partitions. 

Let $p_d (n)$ and $p_a (n)$ equal $\# P_d (n)$ and $\# P_a (n)$, respectively. Hardy and Ramanujan \cite[\S 7.1]{HR} found the following asymptotic formula for the number of partitions of $n$ into distinct parts:
\[p_d (n) \sim \frac{1}{4 \sqrt[4]{3} n^{3/4}} \exp\left(\pi \sqrt{\frac{n}{3}}\right).\]
As an application of our techniques, we find an asymptotic formula for the the number of partitions of $n$ into almost consecutive parts:
\[p_a (n) \sim \frac{\pi}{(8 \cdot 3^{3/4}) n^{5/4}} \exp\left(\pi \sqrt{\frac{n}{3}}\right).\]

Finally, we provide the following equivalence for almost consecutive partitions.

\begin{thm} \label{final theorem} The number of almost consecutive partitions of $n$ is equal to the number of triplets $(a, b, r)$ for which
\[n = 1 + 2 + \cdots + (r - 2) + a(r - 1) + br\]
with $a, b$ positive and $r \geq 2$.
\end{thm}

By removing the numbers $1 + 2 + \cdots + r$ from this sum and decreasing $a$ and $b$ by $1$, we can also observe that the number of almost consecutive partitions of $n$ is equal to the number of triplets $(a, b, r)$ with $a, b \geq 0$ and $r \geq 2$ satisfying $n - (r(r + 1)/2) = a(r - 1) + br$.

\section{Consecutive partitions}

From here on, we let $P_c (n)$ be the set of partitions of $n$ into distinct consecutive parts. In order to prove Theorem $5$, we need to use a lemma related to $P_c (n)$. Sylvester \cite[\S 46]{Syl} proved that $\# P_c (n)$ is equal to the number of odd divisors of $n$. Mason \cite{Mas} (see also \cite{HH}) extended this result.

\begin{thm} Let $n$ be an integer with largest odd divisor $m$. The number of ways to express $n$ as a sum of an odd number of consecutive numbers is equal to the number of divisors of $m$ which are $< \sqrt{2n}$ and the number of ways to use an even number of consecutive numbers is equal to the number of divisors of $m$ which are $> \sqrt{2n}$.
\end{thm}

The proof of the previous result relies on classifying the solutions to the equations
\[n = (a - k) + (a - k + 1) + \cdots + a + \cdots + (a + k),\]
\[n = (a - k + 1) + (a - k + 2) + \cdots + a + (a + 1) + \cdots + (a + k).\]
We use a similar technique to prove Theorem $5$.

\begin{thm} Let $n$ be an integer with largest odd divisor $m$. We have
\[\sum_{\lambda \in P_c (n)} (-1)^{\#(\lambda)} s(\lambda) = -\frac{1}{2} (\sigma(n) + \#\{d | m : d < \sqrt{2n}\} - \#\{d | m : d > \sqrt{2n}\}).\]
\end{thm}

\begin{proof} Let $n = 2^b m$ with $m$ odd. The representations of $n$ as a sum of an odd number of consecutive terms have the form
\[n = (a - k) + (a - k + 1) + \cdots + a + \cdots + (a + k).\]
In this case, $n = (2k + 1)a$. Letting $d = 2k + 1$, we see that $d$ is odd and $d < \sqrt{2n}$. In fact, $d$ can be any odd divisor of $n$ which is $< \sqrt{2n}$. A partition of this form has odd length and smallest element
\[a - k = \frac{n}{d} - \frac{d - 1}{2}.\]

We now consider the partitions of even length. In this case,
\[n = (a - k + 1) + (a - k + 2) + \cdots + a + (a + 1) + \cdots + (a + k).\]
This time, we have $n = (2a + 1)k$. Letting $d = 2a + 1$, we still have an odd value of $d$ but now $d > \sqrt{2n}$. The smallest element of the partition is now
\[a - k + 1 = \frac{d + 1}{2} - \frac{n}{d}.\]

Combining these arguments gives us
\begin{eqnarray*}
\sum_{\lambda \in P_c (n)} (-1)^{\# (\lambda)} s(\lambda) & = & \sum_{\substack{d | m \\ d > \sqrt{2n}}} \left(\frac{d + 1}{2} - \frac{n}{d}\right) - \sum_{\substack{d | m \\ d < \sqrt{2n}}} \left(\frac{n}{d} - \frac{d - 1}{2}\right) \\
& = & -n\sum_{d | m} \frac{1}{d} + \frac{1}{2} \sum_{d | m} d + \frac{1}{2} \left(\sum_{\substack{d | m \\ d < \sqrt{2n}}} 1 - \sum_{\substack{d | m \\ d > \sqrt{2n}}} 1\right) \\
& = & -\frac{n \sigma(m)}{m} + \frac{\sigma(m)}{2} - \frac{\#\{d | m : d < \sqrt{2n}\} - \#\{d | m : d > \sqrt{2n}\}}{2} \\
& = & -\left(2^b - \frac{1}{2}\right) \sigma(m) - \frac{\#\{d | m : d < \sqrt{2n}\} - \#\{d | m : d > \sqrt{2n}\}}{2}.
\end{eqnarray*}
Note that $\sigma(n) = (2^{b + 1} - 1) \sigma(m)$. Hence,
\[\sum_{\lambda \in P_c (n)} (-1)^{\#(\lambda)} s(\lambda) = -\frac{1}{2} (\sigma(n) + \#\{d | m : d < \sqrt{2n}\} - \# \{d | m : d > \sqrt{2n}\}). \qedhere\]
\end{proof}

\section{A bijection between partitions}

Recall that the partition $\lambda$ is \emph{almost consecutive} if $\lambda = (\lambda_1, \lambda_2, \ldots, \lambda_k)$ with $\lambda_i < \lambda_{i + 1}$ for all $i < k$ and $\lambda_{i + 1} = \lambda_i + 1$ for all $i > 1$. To do this, we find a map between certain partitions of $n$ into distinct parts and partitions of $n - 1$ into distinct parts. Both the smallest part and number of parts of a partition are invariant under this map, allowing us to evaluate the sums of the smallest parts of certain sets of partitions.

For a given integer $m$, we let $P_d (m)$ be the set of partitions of $m$ into distinct parts and $P_c (m)$ be the subset of $P_d (m)$ in which the parts are consecutive. In addition, $P_a (m)$ is the set of almost consecutive partitions of $m$.

Fix an integer $n$. We define the map $g: P_d (n + 1) \backslash P_c (n + 1) \to P_d (n)$ as follows. Let $\lambda = (\lambda_1, \lambda_2, \ldots, \lambda_k) \in P_d (n + 1) \backslash P_c (n + 1)$ with $\lambda_i < \lambda_{i + 1}$ for all $i < k$. Let $m$ be the largest integer satisfying $\lambda_m > \lambda_{m - 1} + 1$. (The variable $m$ is well-defined because $\lambda$ does not consist of consecutive parts.) To find $g(\lambda)$, simply replace $\lambda_m$ with $\lambda_m - 1$. So,
\[g(\lambda_1, \lambda_2, \ldots, \lambda_k) = (\lambda_1, \lambda_2, \ldots, \lambda_{m - 1}, \lambda_m - 1, \lambda_{m + 1}, \ldots, \lambda_k).\]
This map is well-defined because $\lambda_m - 1 > \lambda_{m - 1}$, ensuring that the parts are still distinct.

In order to use $g$, we classify the number of inverse images a partition $\lambda = (\lambda_1, \lambda_2, \ldots, \lambda_k) \in P_d (n)$ can have under $g$. First note that the singleton partition $(n)$ has no inverses because $g$ only acts on partitions with multiple elements and $g$ does not change the number of elements of a partition. However, if $\lambda$ is any \emph{other} element of $P_d (n)$, then $g^{-1} (\lambda)$ is non-empty. Specifically,
\[(\lambda_1, \lambda_2, \ldots, \lambda_{k - 1}, \lambda_k + 1) \in g^{-1} (\lambda).\]

\begin{lemma} The length and smallest part of a partition remain invariant under $g$.
\end{lemma}

\begin{proof} In order to find $g(\lambda)$, we replace $\lambda_m$ with $\lambda_m - 1$ for some $m > 1$. Because $m > 1$, $\lambda_1$ remains constant. In addition, $\lambda_m - 1 > 0$, ensuring that we do not remove any rows.
\end{proof}

\begin{lemma} If $\lambda \in P_a (n) \backslash \{(n)\}$, then $\# g^{-1} (\lambda) = 1$.
\end{lemma}

\begin{proof} In a given inverse of $\lambda$, we have to replace $\lambda_m$ with $\lambda_m + 1$ for some $m > 1$. Suppose $m < k$. Because $\lambda$ is almost consecutive, $\lambda_{m + 1} = \lambda_m + 1$. If we replace $\lambda_m$ with $\lambda_m + 1$, then our new partition no longer consists of distinct parts. Hence, $m$ cannot be less than $k$, which implies that the only solution is $m = k$.
\end{proof}

\begin{lemma} Let $\lambda \in P_d (n) \backslash \{(n)\}$. If $\lambda_k - \lambda_{k - 1} > 2$, then $\# g^{-1} (\lambda) = 1$.
\end{lemma}

\begin{proof} Once again, we suppose that $(\lambda_1, \lambda_2, \ldots, \lambda_{m - 1}, \lambda_m + 1, \lambda_{m + 1}, \ldots, \lambda_k) \in g^{-1} (\lambda)$ and show that $m$ must equal $k$. By assumption, $m > 1$. Suppose $g(\pi) = \lambda$. Because $\lambda_k - \lambda_{k - 1} > 2$, we have $\pi_k - \pi_{k - 1} > 1$. So, when we apply $g$ to $\pi$, we must replace $\pi_k$ with $\pi_k - 1$. Therefore, $m = k$.
\end{proof}

\begin{lemma} Let $\lambda \in P_d (n) \backslash \{(n)\}$ be a partition not discussed in the previous two lemmata. Then, $\# g^{-1} (\lambda) = 2$.
\end{lemma}

\begin{proof} There are two cases based on the size of $\lambda_k - \lambda_{k - 1}$. First, suppose that $\lambda_k - \lambda_{k - 1} = 1$. We already know that $(\lambda_1, \lambda_2, \ldots, \lambda_{k - 1}, \lambda_k + 1) \in g^{-1} (\lambda)$. Additionally, the largest value of $m$ satisfying $\lambda_m - \lambda_{m - 1} > 1$ lies in $[3, k - 1]$. So,
\[g(\lambda_1, \lambda_2, \ldots, \lambda_{m - 1} + 1, \lambda_m, \ldots, \lambda_k) = \lambda.\]
In addition, if we apply $g$ to any inverse of $\lambda$, then we can only change the $r$th row if $r \geq m - 1$. Thus, $\# g^{-1} (\lambda) = 2$.

Now suppose $\lambda_k - \lambda_{k - 1} = 2$. The only possible inverses involve increasing $\lambda_k$ or $\lambda_{k - 1}$ by $1$.
\end{proof}

\section{Applying the bijection}

Now that we have established several notable properties of $g$, we can use them to prove Theorem $6$. We split $P_d (n) \backslash \{(n)\}$ into several subsets. Let $\mathcal{S}_1$ be the set of almost consecutive partitions with length greater than $1$ for which $\lambda_k - \lambda_{k - 1} \leq 2$. Let $\mathcal{S}_2$ be the set of all $\lambda = (\lambda_1, \lambda_2, \ldots, \lambda_k) \in P_d (n) \backslash \{(n)\}$ with $\lambda_k - \lambda_{k - 1} > 2$. Finally, let $\mathcal{S}_3$ be all other partitions.

By definition,
\[\sum_{\lambda \in P_d (n)} f(\lambda) = \sum_{\lambda \in \mathcal{S}_1} f(\lambda) + \sum_{\lambda \in \mathcal{S}_2} f(\lambda) + \sum_{\lambda \in \mathcal{S}_3} f(\lambda) + f((n)).\]
In addition, we can rewrite the $\mathcal{S}_2$ sum. Define the map $h: \mathcal{S}_2 \to P_d (n - 2) \backslash \{(n - 2)\}$ as
\[h(\lambda_1, \lambda_2, \ldots, \lambda_k) = (\lambda_1, \lambda_2, \ldots, \lambda_{k - 1}, \lambda_k - 2).\]
Then, $h$ is a bijection which preserves the length and smallest part of a partition. So,
\[\sum_{\lambda \in \mathcal{S}_2} f(\lambda) = \sum_{\lambda \in P_d (n - 2) \backslash \{(n - 2)\}} f(\lambda) = \sum_{\lambda \in P_d (n - 2)} f(\lambda) - f((n - 2)).\]
Plugging this back into our sum over all $\lambda \in P_d (n)$ gives us
\[\sum_{\lambda \in \mathcal{S}_1} f(\lambda) + \sum_{\lambda \in \mathcal{S}_3} f(\lambda) = \sum_{\lambda \in P_d (n)} f(\lambda) - \sum_{\lambda \in P_d (n - 2)} f(\lambda) - f((n)) + f((n - 2)).\]

We now relate these quantities to subsets of $P_d (n + 1)$. By the lemmata in the previous section, the elements of $\mathcal{S}_1$ and $\mathcal{S}_2$ have one inverse under $g$, while the elements of $\mathcal{S}_3$ have two. In addition, these inverses have the same number of parts and smallest part. Therefore,
\begin{eqnarray*}
\sum_{\lambda \in P_d (n + 1)} f(\lambda) & = & \sum_{\lambda \in g^{-1} (\mathcal{S}_1)} f(\lambda) + \sum_{\lambda \in g^{-1} (\mathcal{S}_2)} f(\lambda) + \sum_{\lambda \in g^{-1} (\mathcal{S}_3)} f(\lambda) + \sum_{\lambda \in P_c (n + 1)} f(\lambda) \\
& = & \sum_{\lambda \in \mathcal{S}_1} f(\lambda) + \sum_{\lambda \in \mathcal{S}_2} f(\lambda) + 2 \sum_{\lambda \in \mathcal{S}_3} f(\lambda) + \sum_{\lambda \in P_c (n + 1)} f(\lambda).
\end{eqnarray*}
Using our previous results, we can evaluate a few of these sums. We already established that
\[\sum_{\lambda \in \mathcal{S}_2} f(\lambda) = \sum_{\lambda \in P_d (n - 2)} f(\lambda) - f((n - 2)),\]
\[\sum_{\lambda \in \mathcal{S}_3} f(\lambda) = \sum_{\lambda \in P_d (n)} f(\lambda) - \sum_{\lambda \in P_d (n - 2)} f(\lambda) - \sum_{\lambda \in \mathcal{S}_1} f(\lambda) - f((n)) + f((n - 2)).\]
Plugging these formulas back into the sum over all $\lambda \in P_d (n + 1)$ gives us
\begin{eqnarray*}
\sum_{\lambda \in P_d (n + 1)} f(\lambda) & = & \sum_{\lambda \in \mathcal{S}_1} f(\lambda) + \sum_{\lambda \in P_d (n - 2)} f(\lambda) - f((n - 2)) \\
& & + 2\left(\sum_{\lambda \in P_d (n)} f(\lambda) - \sum_{\lambda \in P_d (n - 2)} f(\lambda) - \sum_{\lambda \in \mathcal{S}_1} f(\lambda) - f((n)) + f((n - 2))\right) \\
& & + \sum_{\lambda \in P_c (n + 1)} f(\lambda) \\
& = & 2 \sum_{\lambda \in P_d (n)} f(\lambda) + \sum_{\lambda \in P_c (n + 1)} f(\lambda) -\sum_{\lambda \in \mathcal{S}_1} f(\lambda) - \sum_{\lambda \in P_d (n - 2)} f(\lambda) \\
& & + f((n - 2)) - 2f((n)),
\end{eqnarray*}
which implies that
\begin{eqnarray*}
\sum_{\lambda \in \mathcal{S}_1} f(\lambda) & = & 2\sum_{\lambda \in P_d (n)} f(\lambda) + \sum_{\lambda \in P_c (n + 1)} f(\lambda) - \sum_{\lambda \in P_d (n + 1)} f(\lambda) - \sum_{\lambda \in P_d (n - 2)} f(\lambda) \\
& & + f((n - 2)) - 2f((n)).
\end{eqnarray*}

Our original goal was to sum $f(\lambda)$ over all almost consecutive $\lambda$ with sum $n$. By definition,
\[\mathcal{S}_1 = \{|\lambda| = n : \lambda \textrm{ almost consecutive}, L(\lambda) = k > 1, \lambda_k - \lambda_{k - 1} \leq 2\}.\]
We now classify the almost consecutive partitions of $n$ which do not belong to $\mathcal{S}_1$. If $k > 2$, then $\lambda_k - \lambda_{k - 1} = 1$ because $\lambda_2, \lambda_3, \ldots, \lambda_k$ are all consecutive. The only almost consecutive partitions of $n$ which do not belong to $\mathcal{S}_1$ are $((n))$ and the partitions of the form $(i, n - i)$ with $n - 2i > 2$. In addition, if $n - 2i > 2$, then $i < (n/2) - 1$. Hence,
\[\sum_{\lambda \in P_a (n)} f(\lambda) = \sum_{\lambda \in \mathcal{S}_1} f(\lambda) + f((n)) + \sum_{i = 1}^{\lfloor (n - 3)/2 \rfloor} f((i, n - i)).\]
Substituting our formula for the sum of $f(\lambda)$ over all $\lambda \in \mathcal{S}_1$ into the previous formula gives us
\begin{eqnarray*}
\sum_{\lambda \in P_a (n)} f(\lambda) & = & 2\sum_{\lambda \in P_d (n)} f(\lambda) + \sum_{\lambda \in P_c (n + 1)} f(\lambda) - \sum_{\lambda \in P_d (n + 1)} f(\lambda) - \sum_{\lambda \in P_d (n - 2)} f(\lambda) \\
& & + \sum_{i = 1}^{\lfloor (n - 3)/2 \rfloor} f((i, n - i)) + f((n - 2)) - f((n)).
\end{eqnarray*}

\section{Special cases}

In this section, we apply Theorem $6$ to specific functions $f$. For each of these functions, we can sum $f(\lambda)$ over all $\lambda \in P_d (n)$ and $\lambda \in P_c (n)$, allowing us to find the corresponding sum for $\lambda \in P_a (n)$ as well. Recall that $p_a (n) = \# P_a (n)$ and $p_d (n) = \# P_d (n)$.

\begin{thm} For all $n \geq 3$, we have
\[p_a (n) = 2p_d (n) - p_d (n + 1) - p_d (n - 2) + \#\{d | n + 1 : d \textrm{ odd}\} + \left \lfloor \frac{n - 3}{2} \right \rfloor.\]
\end{thm}

\begin{proof} To obtain this result, we simply use the function $f(\lambda) = 1$. In addition, we recall that $\# P_c (n)$ is simply the number of odd divisors of $n$.
\end{proof}

\begin{thm} We have
\begin{eqnarray*}
\sum_{\lambda \in P_a (n)} (-1)^{\# (\lambda)} & = & 2h(n) - h(n + 1) - h(n - 2) - \# \{d | n + 1 : d < \sqrt{2(n + 1)}, d \textrm{ odd}\} \\
& & + \# \{d | n + 1 : d > \sqrt{2(n + 1)}, d \textrm{ odd}\} + \left\lfloor \frac{n - 3}{2} \right\rfloor,
\end{eqnarray*}
with
\[h(n) = \left\{\begin{array}{ll}
(-1)^k, & \textrm{if } n = k(3k - 1)/2, \\
0, & \textrm{otherwise}.
\end{array}\right.\]
\end{thm}

\begin{proof} This is simply the $f(\lambda) = (-1)^{\# (\lambda)}$ case of Theorem $5$. For the sums over the sets of distinct parts, we apply Euler's Pentagonal Number Theorem. For the sums over consecutive parts, we apply Theorem $8$.
\end{proof}

Note that this quantity is asymptotic to $n/2$ as $n \to \infty$.

\begin{thm} We have
\begin{eqnarray*}
\sum_{\lambda \in P_a (n)} (-1)^{\# (\lambda)} s(\lambda) & = & d(n + 1) + d(n + 2) - 2d(n) + T\left(\left\lfloor \frac{n - 3}{2} \right\rfloor\right) + 2 \\
& & -\frac{1}{2} (\sigma(n + 1) + \#\{d | n + 1 : d < \sqrt{2(n + 1)}, d \textrm{ odd}\} \\
& & - \# \{d | n + 1 : d > \sqrt{2(n + 1)}, d \textrm{ odd}\}).
\end{eqnarray*}
\end{thm}

Using Theorem $13$, we can obtain an asymptotic formula for the number of partitions of $n$ into almost consecutive parts.

\begin{thm} We have
\[p_a (n) \sim \frac{\pi}{(8 \cdot 3^{3/4}) n^{5/4}} \exp\left(\pi \sqrt{\frac{n}{3}}\right).\]
\end{thm}

\begin{proof} Theorem $13$ implies that
\[p_a (n) = 2p_d (n) - p_d (n + 1) - p_d (n - 2) + O(n).\]
A classic theorem of Hardy and Ramanujan \cite{HR} states that
\[p_d (n) = \left(1 + O\left(\frac{1}{n}\right)\right) \frac{1}{4 \sqrt[4]{3} n^{3/4}} \exp\left(\pi \sqrt{\frac{n}{3}}\right).\]
For a pair of real numbers $C, k$, we may observe that
\[(n + k)^{C} = n^C (1 + (1 + o(1))(k/n))^C = n^C (1 + (C + o(1)) (k/n)),\]
\[e^{C\sqrt{n + k}} = e^{C\sqrt{n}} e^{(C + o(1)) \sqrt{n} (k/(2n))} = e^{C\sqrt{n}} e^{(C/2 + o(1)) k/\sqrt{n}} = e^{C \sqrt{n}} \left(1 + \left(\frac{C}{2} + o(1)\right) \frac{k}{\sqrt{n}}\right),\]
as $n \to \infty$. Hence,
\begin{eqnarray*}
p_d (n + 1) & = & \frac{1}{4 \sqrt[4]{3} n^{3/4}} \exp\left(\pi \sqrt{\frac{n}{3}}\right) + (1 + o(1)) \frac{\pi}{8 \cdot 3^{3/4} n^{5/4}} \exp\left(\pi \sqrt{\frac{n}{3}}\right), \\
p_d (n - 2) & = & \frac{1}{4 \sqrt[4]{3} n^{3/4}} \exp\left(\pi \sqrt{\frac{n}{3}}\right) - (1 + o(1)) \frac{\pi}{4 \cdot 3^{3/4} n^{5/4}} \exp\left(\pi \sqrt{\frac{n}{3}}\right),
\end{eqnarray*}
Putting everything together gives us our desired bound.
\end{proof}

\section{Another partition equivalence}

Finally, we prove Theorem \ref{final theorem}. Suppose we write $n$ as a sum of $\ell + 1$ terms in which the last $\ell$ terms are consecutive. Let the smallest term be $k$ and the gap between the smallest two terms $m$. Then,
\[n = k + ((k + m) + (k + m + 1) + \cdots + (k + m + \ell - 1)) = k(\ell + 1) + m\ell + (\ell(\ell - 1)/2).\]
The corresponding generating function is
\[\sum_{k, \ell, m \geq 1} q^{k(\ell + 1) + m\ell + (\ell(\ell - 1)/2)} = \sum_{\ell = 1}^\infty q^{\ell(\ell - 1)/2} \left(\sum_{k = 1}^\infty q^{k(\ell + 1)} \sum_{m = 1}^\infty q^{m\ell}\right).\]
The last two sums are both geometric series, making them straightforward to evaluate. Our sum is now
\[\sum_{\ell = 1}^\infty q^{\ell(\ell - 1)/2} \cdot \frac{q^{\ell + 1}}{1 - q^{\ell + 1}} \cdot \frac{q^\ell}{1 - q^\ell} = \sum_{\ell = 1}^\infty \frac{q^{(\ell + 1)(\ell + 2)/2}}{(1 - q^\ell)(1 - q^{\ell + 1})}.\]

We now discuss a combinatorial interpretation of this quantity. Letting $r = \ell + 1$ allows us to rewrite the generating function as
\[\sum_{r = 2}^\infty q^{1 + 2 + \cdots + r} (1 + q^{r - 1} + q^{2(r - 1)} + \cdots)(1 + q^r + q^{2r} + \cdots).\]
This sum counts partitions of a number $n$ into numbers $1, 2, \ldots, r$ with $r \geq 2$ in which $1, 2, \ldots, r - 2$ occur exactly once and $r - 1$ and $r$ simply need to occur at least once.

\section{Competing Interests and Data Statements}

The authors did not receive support from any organization for the submitted work.

The authors have no data sets to make available.

\end{document}